\documentclass[12pt]{article}
\usepackage{amsmath,amssymb,amsthm,amsfonts,epsfig}

\theoremstyle{definition}
\newtheorem{thm}{Theorem}[section]
\newtheorem{prop}[thm]{Proposition}
\newtheorem{cor}[thm]{Corollary}
\newtheorem{lem}[thm]{Lemma}

\newtheorem{ex}[thm]{Example}

\newtheorem*{ex*}{Example}

\newcommand{\pfend}{\hspace{3mm}$\Box$}

\date{}
\begin{document}

\title{On complex symmetric block Toeplitz operators}
\author{Dong-O Kang, Eungil Ko and Ji Eun Lee}

\maketitle \footnotetext{2010 Mathematics Subject Classification;
Primary 47B35, 47A05. %47B15, 47B20.
\par Keywords; Conjugation; Complex symmetric operator; Block Toeplitz operator.
\par This  third author was supported by Basic Science Research Program through the National Research Foundation of Korea(NRF) funded by the Ministry of Education(2016R1D1A1B03931937).  
  The last author was supported by Basic Science Research Program through the National Research Foundation of Korea(NRF) funded by the Ministry of Education, Science and Technology(2019R1A2C1002653).
}
\begin{abstract}

In this paper, we study complex symmetry of Toeplitz operators and  block  Toeplitz operators. In particular,  we give a characterization of complex symmetric block Toeplitz operators with  the special conjugation on the vector-valued Hardy space $H_{{\mathbb C}^2}^2$. As some applications, we provide examples of such operators.
\end{abstract}

% sec 1
\section{Introduction and  Preliminaries}

\par The study of complex symmetric operators is the interaction of the fields of operator theory and complex analysis. Recently, many authors have been interested in non-hermitian quantum mechanics
and in the spectral analysis of certain complex symmetric operators (see \cite{BFGJ} and \cite{GRP}). In particular, several authors have been studied an antilinear operator which is the only type of nonlinear operators that are  important in quantum mechanics (see \cite{Al}, \cite{So}, and \cite{We}).
A {\it conjugation} on $\cal H$ is an antilinear operator ${C}: {\cal H}\rightarrow {\cal H}$ with  ${C}^{2}=I$ which satisfies $\langle {C}x, {C}y \rangle=\langle y, x\rangle$ for all $x,y\in {\cal H}$. For the conjugation $C$, there  exists an orthonormal basis $\{e_n\}_{n=0}^{\infty}$ for ${\cal H}$ such that ${C}e_n=e_n$ for all $n$ (see \cite{Ga 3}).
We call an operator $T\in{\cal L(H)}$  {\it complex symmetric} if there exists a conjugation $ C$ on ${\cal H}$ such that $T= {C}T^{\ast}{C}$.
The class of complex symmetric operators includes all normal operators, Hankel operators, truncated Toeplitz operators,  Volterra integration operators, and etc.
We refer the reader to \cite{Ga 3}, \cite{GW}, \cite{JKL}, and \cite{KL} for more details, including historical comments and references.

Let $L^2$ be the Lebesgue (Hilbert) space on the unit circle $\partial{\Bbb D}$ and let $L^{\infty}$ be the Banach space of all essentially bounded functions on $\partial{\Bbb D}$. 
If $f\in L^{2},$ then the function $f$ is expressed as $f(z)=\sum_{n=-\infty}^{\infty}\hat{f}(n)z^{n}$ where $\hat{f}(n)$ denotes the $n$th Fourier coefficient of $f.$ The Hilbert Hardy space, denoted by $H^2$, consists of all functions $f$  analytic on the open unit disk $\Bbb D$ with the power series representation $$f(z)=\sum_{n=0}^{\infty}a_nz^n\ \mbox{with}\ \sum_{n=0}^{\infty}|a_n|^2<\infty.$$
For any $\varphi\in L^{\infty}$, the {\it Toeplitz operator} $T_{\varphi}:{H^2}\rightarrow{H^2}$ is defined by the formula
$$T_{\varphi}{f}=P(\varphi f)\ \mbox{for}\ f\in{H^2}$$ where $P$ denotes the orthogonal projection of $L^2$ onto $H^2$. 

Let $L_{{\Bbb C}^2}^2=L^2\otimes{\Bbb C}^2$, $H_{{\Bbb C}^2}^2=H^2\otimes{\Bbb C}^2$, and let $M_2$ be the set of $2 \times 2$ complex matrices.
For a matrix-valued function $\Phi\in L^{\infty}_{M_2}=L^{\infty}\otimes M_2$, the {\it block Toeplitz operator with symbol $\Phi$} is the operator $T_{\Phi}$
on the vector-valued Hardy space $H_{{\Bbb C}^2}^2$ of the unit disk defined by
$$T_{\Phi}h=P_2(\Phi h), \  h\in H_{{\Bbb C}^{2}}^2$$ where $P_2$ denotes the orthogonal projection of $L_{{\Bbb C}^2}^2$ onto $H_{{\Bbb C}^2}^2$ (see  \cite{HKL} and  \cite{Gu}). In particular,
 if  $\Phi=\begin{pmatrix} \varphi_{1} &\varphi_{2} \cr \varphi_{3} &\varphi_{4} \end{pmatrix}$ where $\varphi_{j}\in L^{\infty}$ for $j=1,2,3,4$,  then the block Toeplitz operator has the following representation; $$
T_{\Phi}=\begin{pmatrix}T_{\varphi_{1}} & T_{\varphi_{2}}\cr
T_{\varphi_{3}} & T_{\varphi_{4}}\end{pmatrix}.$$

The study of  Toeplitz operators and block Toeplitz operators provides deep and important connections with various problems in the field of physics (see \cite{BFGJ} and \cite{GRP}).
In particular, the spectral theory of Toeplitz operators is applied to the research in  quantum mechanics (see \cite{P}). In addition, the study of the theory of block Toeplitz determinants plays a
significant role in the study of high-temperature superconductivity (see \cite{BE}).
In 1960's, A. Brown and P. Halmos \cite{BH} proved that $T_{\varphi}$ is normal if and only if $\varphi=\alpha+\beta\rho$ where  $\rho$ is a real-valued function in $L^{\infty}$ and $\alpha, \beta\in{\Bbb C}$. In general, $T_{\varphi}$ may not be a complex symmetric operator.
In 2014, K. Guo and S. Zhu \cite{GZ} have raised the following interesting question; characterize a complex symmetric Toeplitz operator on $H^2$. Recently,  the authors in \cite{KL} gave  a characterization of a complex symmetric Toeplitz operator $T_{\varphi}$ on $H^2$ with some conjugations.
In view of this, we are interested in the following question;

$${\mbox{\it When is a block Toeplitz operator $T_{\Phi}$ complex symmetric?}}$$
These results give us some motivations to study complex symmetric block Toeplitz operators.

\medskip
In this paper, we study  complex symmetry of Toeplitz operators and block Toeplitz operators.
In particular, we provide a characterization of complex symmetric block Toeplitz operators with the  special conjugation. From this, we give examples of normal complex symmetric block Toeplitz operators.

\bigskip

\section{Remark on complex symmetry of a Toeplitz operator}
In this section, we give a characterization of complex symmetric Toeplitz operator with the special conjugation on the Hardy space $H^2$.
Recall that if $C_{\mu,\lambda}: H^2\rightarrow H^2$ is defined by $$C_{\mu,\lambda}f(z)=\mu\overline{f({\lambda}\overline{z})}$$ for all $\lambda,\mu\in \partial {\Bbb D}$, then $C_{\mu,\lambda}$ is the conjugation operator on $H^2$ from \cite{KL}.
Let  $P$ denote the orthogonal projection of $L^2$ onto $H^2$. 
Then the operator $C_{\mu,\lambda}$ can also be defined on $L^{2}.$ Note that
\begin{equation}
C_{\mu,\lambda}P=PC_{\mu,\lambda},\label{one}%
\end{equation}
since for $n\geq0$%
\[
PC_{\mu,\lambda}z^{n}=P\mu\overline{\lambda}^{n}z^{n}=\mu\overline{\lambda}^{n}%
z^{n}=C_{\mu,\lambda}z^{n}=C_{\mu,\lambda}Pz^{n}%
\]
and for $n<0,PC_{\mu,\lambda}\overline{z}^{-n}=0=C_{\mu,\lambda}P\overline
{z}^{-n}.$ We start with the following lemmas.

%2.1
\begin{lem}\label{lem_1} {\it
On $H^{2},$ for any operator $T,$ $C_{\mu,\lambda}TC_{\mu,\lambda
}=C_{1,\lambda}TC_{1,\lambda}$.  Hence, the operator $T$ is complex symmetric with the conjugation $C_{\mu
,\lambda}$ if and only if $T$ is complex symmetric with the conjugation
$C_{1,\lambda}.$}
\end{lem}

\begin{proof}
Since $C_{\mu,\lambda}=\mu I\cdot C_{1,\lambda}=C_{1,\lambda}%
\cdot\overline{\mu}I,$ it follows that%
\begin{align*}
C_{\mu,\lambda}TC_{\mu,\lambda}  & =\mu I\cdot C_{1,\lambda}\cdot T\cdot\mu
I\cdot C_{1,\lambda}\\
& =\mu I\cdot C_{1,\lambda}\cdot\mu I\cdot T\cdot C_{1,\lambda}\\
& =C_{1,\lambda}\cdot\overline{\mu}I\cdot\mu I\cdot T\cdot C_{1,\lambda}\\
& =C_{1,\lambda}\cdot T\cdot C_{1,\lambda}.
\end{align*}
Therefore the operator $T$ is complex symmetric with the conjugation $C_{\mu
,\lambda}$ if and only if $T$ is complex symmetric with the conjugation
$C_{1,\lambda}.$
\end{proof}
%2.2
\begin{lem}\label{lem_2}
{\it On $H^{2},$ for any Toeplitz operator $T_{\varphi},$
\[
C_{1,\lambda}T_{\varphi(z)}C_{1,\lambda}=T_{C_{1,\lambda}\varphi
(z)}=T_{\overline{\varphi(\lambda\overline{z})}}.
\]}
\end{lem}

\begin{proof}
For any $h\in H^{2},$  we have%
\begin{align*}
C_{1,\lambda}T_{\varphi(z)}C_{1,\lambda}h(z)  & =C_{1,\lambda}T_{\varphi
(z)}\overline{h(\lambda\overline{z})}\\
& =C_{1,\lambda}P\left[  \varphi(z)\overline{h(\lambda\overline{z})}\right]
\\
& =PC_{1,\lambda}\left[  \varphi(z)\overline{h(\lambda\overline{z})}\right]
\\
& =P\left[  \overline{\varphi(\lambda\overline{z})}h(z)\right]  =T_{\overline
{\varphi(\lambda\overline{z})}}h(z),
\end{align*}
where the third equality follows from (\ref{one}).
\end{proof}

We prove one of our main results.
%2.3
\begin{thm}\label{thm1}
{\it $T_{\varphi}$ is complex symmetric with  the conjugation $C_{1,\lambda}$ if and
only if $$\varphi(z)=\varphi(\lambda\overline{z})\ \mbox{ on}\ \left\vert z \right\vert
=1.$$}
\end{thm}

\begin{proof}
By definition, $T_{\varphi}$ is complex symmetric with the conjugation
$C_{1,\lambda}$ if and only if%
\[
T_{\varphi(z)}=C_{1,\lambda}T_{\varphi(z)}^{\ast}C_{1,\lambda}=C_{1,\lambda}T_{\overline{\varphi(z)}}C_{1,\lambda}.%
\]
By Lemma \ref{lem_2}, we have
\[
C_{1,\lambda}T_{\overline{\varphi(z)}}C_{1,\lambda}=T_{\varphi(\lambda
\overline{z})}.
\]
So we obtain $T_{\varphi(\lambda\overline{z})}=T_{\varphi(z)}$ and
$\varphi(z)=\varphi(\lambda\overline{z}).$
\end{proof}

\
\medskip

Let us recall that $\varphi(z)=\sum_{n=-\infty}^{\infty}\hat{\varphi}(n)e_n$ belong to $L^{\infty}$.
Set $$\varphi_{+}=S^{\ast}T_{\varphi}e_0=\sum_{n=0}^{\infty}\hat{\varphi}(n+1)e_n\ \mbox{and}\  \varphi_{-}=S^{\ast}T_{\varphi}^{\ast}e_0=\sum_{n=0}^{\infty}\overline{\hat{\varphi}(-n-1)}e_n$$
where  $S$ denotes the unilateral shift on $H^2$.
%2.4
\begin{cor}
{\it Let  $S$ denote the unilateral shift on $H^2$.
If $T_{\varphi}$ is a complex symmetric operator with the conjugation $C_{1,\lambda}$, then $\varphi_{-}(z)= \overline{\varphi_{+}(\lambda\overline{z})}$  with $|\lambda|=1$. Moreover,  $ \hat{\varphi}(-n-1)={\lambda}^{n+1}\hat{\varphi}(n+1)$ for all $n=0,1,2,\cdots$.}
\end{cor}

{\it Proof.} By Theorem \ref{thm1}, if we write
$$\begin{cases}
\varphi(z)=\varphi_{+}(z)+\varphi_0+\overline{\varphi_{-}(z)}\ \mbox{and}\\
\varphi(\lambda\overline{z})=\varphi_{+}(\lambda\overline{z})+\varphi_0+\overline{\varphi_{-}(\lambda\overline{z}}),
\end{cases}
$$ then  $\varphi(z)=\varphi(\lambda\overline{z})$
if and only if $\varphi_{-}(z)=\overline{\varphi_{+}(\lambda\overline{z})}$, which in terms of Fourier coefficients is
 $\hat{\varphi}(-n-1)={\lambda}^{n+1}\hat{\varphi}(n+1)$
for all $n\geq 0$ and $\left\vert \lambda \right\vert
=1.$
\medskip

As a consequence of Theorem \ref{thm1}, we recapture the following corollary.
%2.5
\begin{cor}(\cite[Corollary 2.6]{KL})\label{inf-CSO-0}{\it
{\em (i)} $T_{\varphi}$ is complex symmetric with the conjugation $C_{1,1}$ if and only if $$\varphi(z)=\varphi_0+2\sum_{n=1}^{\infty}\hat{\varphi}(n)Re\{z^n \}.$$
{\em (ii)}  $T_{\varphi}$ is complex symmetric with  the conjugation $C_{1,-1}$ if and only if $$\varphi(z)=\varphi_0+2\sum_{k=1}^{\infty}\hat{\varphi}(2k)Re\{z^{2k} \}+2i\sum_{k=1}^{\infty}\hat{\varphi}(2k-1)Im\{z^{2k-1} \}.$$}
\end{cor}

Remark that  if $\varphi(z)=\phi(z)+\alpha+\phi(\overline{z})$ and $\psi(z)=\phi(z)+\beta+\phi(-\overline{z})$ for $\phi\in zH^{2}$ and $\alpha,\beta\in{\Bbb C}$, then $T_{\varphi}$ and $T_{\psi}$ are complex symmetric operators on $H^2$ (\cite{KL}).
\medskip

%2.6
\begin{prop}\label{CSO} {\it Let  $\{e_n\}$ be an orthonormal basis of $H^2$ and let $C$ be the conjugation on $H^2$ with $Ce_n=e_n$.
If   $T_{\varphi}$ is a complex symmetric operator with the conjugation $C$ and $CSC=S$ where  $S$ denotes the unilateral shift on $H^2$, then   $\varphi_{-}={\lambda}C\varphi_{+}$ and $Ce_0=\overline{\lambda}e_0$  for some $\lambda\in{\partial {\Bbb D}}.$}
\end{prop}

{\it Proof.}  Suppose that $T_{\varphi}$ is a Toeplitz operator with a symbol $\varphi\in L^{\infty}$.
  If
$\varphi_{+}=S^{\ast}T_{\varphi}e_0\ \mbox{and}\ \varphi_{-}=S^{\ast}T_{\varphi}^{\ast}e_0$,
then we have \begin{equation}\label{E_2}
T_{\varphi}S-ST_{\varphi}=e_0\otimes \varphi_{-}\ \mbox{and}\ T_{\varphi}^{\ast}S-ST_{\varphi}^{\ast}=e_0\otimes \varphi_{+}.
\end{equation}
Since $T_{\varphi}^{\ast}=CT_{\varphi}C$, it follows from (\ref{E_2}) that
$$\begin{cases}
T_{\varphi}S-ST_{\varphi}=e_0\otimes \varphi_{-}\\
T_{\varphi}CSC-CSCT_{\varphi}=Ce_0\otimes C\varphi_{+}.
\end{cases}
$$
If $CSC=S$, then $e_0\otimes \varphi_{-}=Ce_0\otimes C\varphi_{+}$ implies that $$e_0=\lambda Ce_0\ \mbox{and}\ \varphi_{-}=\frac{1}{\overline{\lambda}}C\varphi_{+}$$ for some $\lambda\in{\Bbb C}.$
Hence $\varphi_{-}={\lambda}C\varphi_{+}$ and  $Ce_0=\overline{\lambda}e_0$  for some $\lambda\in{\partial {\Bbb D}}.$

\bigskip
%%%%%%%%%%%%%%%

\section{Complex symmetry of a  block Toeplitz operator}
In this section, we  study complex symmetric block Toeplitz operators.  We recently realized the  operator conjugations were rarely studied. So we consider these block Toeplitz operators with  special conjugation 
\begin{equation*}\label{con-B}
\frac{1}{\sqrt{2}}\left(
\begin{array}
[c]{cc}%
C_{\mu,\lambda} & C_{\mu,\lambda}\\
C_{\mu,\lambda} & -C_{\mu,\lambda}%
\end{array}
\right) 
\end{equation*}
on $H_{{\Bbb C}^2}^2$ from \cite{KL1}. From Lemma \ref{lem_1}, we may assume
$$
{\mathcal C}:=\frac{1}{\sqrt{2}}\left(
\begin{array}
[c]{cc}%
C_{1,\lambda} & C_{1,\lambda}\\
C_{1,\lambda} & -C_{1,\lambda}%
\end{array}
\right) 
$$ with $\mu=1.$

%3.1
\begin{thm}\label{main}
{\it $T_{\Phi}=\left(
\begin{array}
[c]{cc}%
T_{\varphi_{1}} & T_{\varphi_{2}}\\
T_{\varphi_{3}} & T_{\varphi_{4}}%
\end{array}
\right)  $ is complex symmetric with the conjugation ${\mathcal C}=\frac{1}{\sqrt{2}}\left(
\begin{array}
[c]{cc}%
C_{1,\lambda} & C_{1,\lambda}\\
C_{1,\lambda} & -C_{1,\lambda}%
\end{array}
\right) $
if and only if 
$$
\begin{cases}
\varphi_1(z)=\frac{1}{2}\sum_{j=1}^{4}\varphi_{j}(\lambda\overline{z})\cr
\varphi_2(z)=\frac{1}{2}\sum_{j=1}^{2}[\varphi_{j}(\lambda\overline{z})-\varphi_{j+2}(\lambda\overline{z})]\cr
\varphi_3(z)=\frac{1}{2}\sum_{j=1}^{4}(-1)^{j+1}\varphi_{j}(\lambda\overline{z})\cr
\varphi_4(z)=\frac{1}{2}\sum_{j=1}^{2}(-1)^{j+1}[\varphi_{j}(\lambda\overline{z})-\varphi_{j+1}(\lambda\overline{z})]
\end{cases}
$$
on $\left\vert z\right\vert
=1$ and $|\lambda|=1$.%
}
\end{thm}

{\it Proof.}
Since $T_{\Phi}$ is complex symmetric with the conjugation ${\cal C}$, it follows that
\begin{equation}\label{E_01_0}
{\cal C}T_{\Phi}^{\ast}
=\frac{1}{\sqrt{2}}\begin{pmatrix}C_{1,\lambda}(T_{\varphi_{1}}^{\ast}+T_{\varphi_{2}}^{\ast}) & C_{1,\lambda}(T_{\varphi_{3}}^{\ast}+T_{\varphi_{4}}^{\ast}) \cr
C_{1,\lambda}(T_{\varphi_{1}}^{\ast}-T_{\varphi_{2}}^{\ast})  & C_{1,\lambda}(T_{\varphi_{3}}^{\ast}- T_{\varphi_{4}}^{\ast}) \end{pmatrix}
\end{equation}
and
\begin{equation}\label{E_02_0}
T_{\Phi}{\cal C}=\frac{1}{\sqrt{2}}\begin{pmatrix} (T_{\varphi_{1}}+T_{\varphi_{2}})C_{1,\lambda} &  (T_{\varphi_{1}}- T_{\varphi_{2}})C_{1,\lambda} \cr
(T_{\varphi_{3}}+T_{\varphi_{4}})C_{1,\lambda}  & (T_{\varphi_{3}}- T_{\varphi_{4}})C_{1,\lambda} \end{pmatrix}.
\end{equation}
 Then (\ref{E_01_0}) and (\ref{E_02_0}) imply that
\begin{equation}\label{E_03-0}
\begin{cases}
T_{\varphi_{1}+\varphi_{2}}=C_{1,\lambda}(T_{\varphi_{1}+\varphi_{2}}^{\ast})C_{1,\lambda}\cr
T_{\varphi_{1}-\varphi_{2}} = C_{1,\lambda}(T_{\varphi_{3}+\varphi_{4}}^{\ast})C_{1,\lambda}\cr
 T_{\varphi_{3}-\varphi_{4}} =C_{1,\lambda}(T_{\varphi_{3}-\varphi_{4}}^{\ast})C_{1,\lambda}\cr
 T_{\varphi_{3}+\varphi_{4}} =C_{1,\lambda}(T_{\varphi_{1}-\varphi_{2}}^{\ast})C_{1,\lambda}.
\end{cases}
\end{equation}
By Theorem \ref{thm1}, we obtain that
$$
\varphi_{1}(z)+\varphi_{2}(z)   =\varphi_{1}(\lambda\overline{z})+\varphi
_{2}(\lambda\overline{z})\ \mbox{and}\
\varphi_{3}(z)-\varphi_{4}(z)   =\varphi_{3}(\lambda\overline{z})-\varphi
_{4}(\lambda\overline{z}).$$
Applying the proof of Lemma \ref{lem_2}, we have
$$
\begin{cases}
C_{1,\lambda}T_{\overline{\varphi_3(z)+\varphi_4(z)}}C_{1,\lambda}=T_{\varphi_3(\lambda
\overline{z})}+T_{\varphi_4(\lambda\overline{z})}\cr
C_{1,\lambda}T_{\overline{\varphi_1(z)-\varphi_2(z)}}C_{1,\lambda}=T_{\varphi_1(\lambda
\overline{z})}-T_{\varphi_2(\lambda\overline{z})}.
\end{cases}
$$
From the second and fourth equations of (\ref{E_03-0}), we get that
$$
\begin{cases}
\varphi_{1}(z)-\varphi_{2}(z)  =\varphi_{3}(\lambda\overline{z})+\varphi
_{4}(\lambda\overline{z})\cr
\varphi_{3}(z)+\varphi_{4}(z)  =\varphi_{1}(\lambda\overline{z})-\varphi
_{2}(\lambda\overline{z}).
\end{cases}
$$
Therefore we have
\begin{equation}\label{E_03-1}
\begin{cases}
\varphi_{1}(z)+\varphi_{2}(z) =\varphi_{1}(\lambda\overline{z})+\varphi
_{2}(\lambda\overline{z}),\\
\varphi_{1}(z)-\varphi_{2}(z) =\varphi_{3}(\lambda\overline{z})+\varphi
_{4}(\lambda\overline{z}),\\
\varphi_{3}(z)-\varphi_{4}(z) =\varphi_{3}(\lambda\overline{z})-\varphi
_{4}(\lambda\overline{z}),\\
\varphi_{3}(z)+\varphi_{4}(z)  =\varphi_{1}(\lambda\overline{z})-\varphi
_{2}(\lambda\overline{z}).
\end{cases}
\end{equation}
Hence  from (\ref{E_03-1}), we get the followings;
\begin{equation}\label{E_03-2}
\begin{cases}
\varphi_1(z)=\frac{1}{2}\sum_{j=1}^{4}\varphi_{j}(\lambda\overline{z})\cr
\varphi_2(z)=\frac{1}{2}\sum_{j=1}^{2}[\varphi_{j}(\lambda\overline{z})-\varphi_{j+2}(\lambda\overline{z})]\cr
\varphi_3(z)=\frac{1}{2}\sum_{j=1}^{4}(-1)^{j+1}\varphi_{j}(\lambda\overline{z})\cr
\varphi_4(z)=\frac{1}{2}\sum_{j=1}^{2}(-1)^{j+1}[\varphi_{j}(\lambda\overline{z})-\varphi_{j+1}(\lambda\overline{z})]
\end{cases}
\end{equation}
 where $|\lambda|=1$.

Conversely, if (\ref{E_03-2}) holds, then $(\ref{E_03-1})$  clearly holds.
Thus we get (\ref{E_03-0}).  By Theorem \ref{thm1} and Lemma \ref{lem_2},  $T_{\Phi}$ is complex symmetric with the conjugation ${\cal C}$.
%The converse implication holds by a similar method.
\pfend

\

\medskip
{The following result is the special case for Theorem \ref{main} with the assumption that $\varphi_{j}(z)$ are trigonometric polynomials.}

%3.2
\begin{cor}\label{finite}{\it Let  ${\Phi}=\begin{pmatrix}{\varphi_{1}} & {\varphi_{2}}\cr
{\varphi_{3}} & {\varphi_{4}}\end{pmatrix}$ where $\varphi_{j}(z)=\sum_{n=-m}^{N}\hat{\varphi}_j(n)z^n$ for $N\geq m>0$ and $j=1,2,3,4$.
Then $T_{\Phi}$ is complex symmetric with the conjugation ${\mathcal C}=\frac{1}{\sqrt{2}}\left(
\begin{array}
[c]{cc}%
C_{1,\lambda} & C_{1,\lambda}\\
C_{1,\lambda} & -C_{1,\lambda}%
\end{array}
\right)$ if and only if
 $N=m$ and
$$
\begin{cases}
\widehat{\varphi_1}(-n)=\frac{1}{2}\sum_{j=1}^{4}\widehat{\varphi_{j}}(n)\lambda^n\cr
\widehat{\varphi_2}(-n)=\frac{1}{2}\sum_{j=1}^{2}[\widehat{\varphi_{j}}(n)-\widehat{\varphi_{j+2}}(n)]\lambda^n\cr
\widehat{\varphi_3}(-n)=\frac{1}{2}\sum_{j=1}^{4}(-1)^{j+1}\widehat{\varphi_{j}}(n)\lambda^n\cr
\widehat{\varphi_4}(-n)=\frac{1}{2}\sum_{j=1}^{2}(-1)^{j+1}[\widehat{\varphi_{j}}(n)-\widehat{\varphi_{j+1}}(n)]\lambda^n
\end{cases}
$$ for all $n=1,2,3,\cdots,N$ and $|\lambda|=1$. }
\end{cor}
\smallskip

{\it Proof.} Applying the proof of Theorem \ref{main}, we get that 
 $T_{\Phi}$ is complex symmetric with the conjugation ${\mathcal C}$ if and only if
 $N=m$ and
\begin{equation}\label{Eq-ex}
\begin{cases}
\widehat{(\varphi_{1}+\varphi_2)}(-n)=\widehat{(\varphi_1+{\varphi}_2)}(n){\lambda}^n\cr
\widehat{(\varphi_{1}-\varphi_2)}(-n)=\widehat{(\varphi_3+{\varphi}_4)}(n){\lambda}^n\cr
\widehat{(\varphi_{3}-\varphi_4)}(-n)=\widehat{(\varphi_3-{\varphi}_4)}(n){\lambda}^n\cr
\widehat{(\varphi_{3}+\varphi_4)}(-n)=\widehat{(\varphi_1-{\varphi}_2)}(n){\lambda}^n,
\end{cases}
\end{equation}
equivalently, 
 $N=m$ and
$$
\begin{cases}
\widehat{\varphi_1}(-n)=\frac{1}{2}\sum_{j=1}^{4}\widehat{\varphi_{j}}(n)\lambda^n\cr
\widehat{\varphi_2}(-n)=\frac{1}{2}\sum_{j=1}^{2}[\widehat{\varphi_{j}}(n)-\widehat{\varphi_{j+2}}(n)]\lambda^n\cr
\widehat{\varphi_3}(-n)=\frac{1}{2}\sum_{j=1}^{4}(-1)^{j+1}\widehat{\varphi_{j}}(n)\lambda^n\cr
\widehat{\varphi_4}(-n)=\frac{1}{2}\sum_{j=1}^{2}(-1)^{j+1}[\widehat{\varphi_{j}}(n)-\widehat{\varphi_{j+1}}(n)]\lambda^n
\end{cases}
$$
 where $|\lambda|=1$.
 \pfend

\medskip

\

We illustrate the following example as some applications of Corollary \ref{finite}.
%3.3
\begin{ex}
Let $${\Phi}=\begin{pmatrix}-1
& -1\cr
-1 & 1\end{pmatrix}\overline{z}+
\begin{pmatrix}1
 & 1\cr
1 & -1\end{pmatrix}z$$
and $${\Psi}=\begin{pmatrix}1
 & -1\cr
-1 & 1\end{pmatrix}\overline{z}+
\begin{pmatrix}1
 & 1\cr
\frac{1}{2} & \frac{3}{2}\end{pmatrix}z.$$
Then (\ref{Eq-ex}) holds  for $T_{\Phi}$. Hence $T_{\Phi}$  is complex symmetric with the conjugation ${\mathcal C}$ (with $\lambda=-1$)  from Corollary \ref{finite}.
However, (\ref{Eq-ex}) does not hold  for $T_{\Psi}$. Hence $T_{\Psi}$  is not complex symmetric with the conjugation ${\mathcal C}$ (with $\lambda=1$) from Corollary \ref{finite}.
\end{ex}

\medskip

Let us recall that $\Phi(z)=\sum_{n=-\infty}^{\infty}\widehat{\Phi}(n)z^n$ where $$\widehat{\Phi}(n)=\begin{pmatrix}\hat{\varphi_{1}}(n) & \hat{\varphi_{2}}(n)\cr
\hat{\varphi_{3}}(n) &\hat{\varphi_{4}}(n)\end{pmatrix}\ \mbox{for all}\ n\in{\Bbb Z}.$$  Put  $\Phi_{+}(z)=\sum_{n=1}^{\infty}\widehat{\Phi}(n)z^n$,  $\Phi_{-}(z)=\sum_{n=1}^{\infty}\overline{\widehat{\Phi}(-n)}z^n,$ and $\Phi_0=\begin{pmatrix}\hat{\varphi_{1}}(0) & \hat{\varphi_{2}}(0)\cr
\hat{\varphi_{3}}(0) &\hat{\varphi_{4}}(0)\end{pmatrix}$.
Then $\Phi_{+}(z),\Phi_{-}(z) \in H^{\infty}_{M_2}$ and $$\Phi_{-}^{\ast}(z)=(\sum_{n=1}^{\infty}\overline{\widehat{\Phi}(-n)}z^n)^{\ast}=\sum_{n=1}^{\infty}{\widehat{\Phi}(-n)}\overline{z}^n=\sum_{n=-\infty}^{-1}{\widehat{\Phi}(n)}{z}^n$$ where $\ast$ denotes the complex conjugate.
Hence $\Phi(z)=\Phi_{+}(z)+\Phi_0+\Phi_{-}^{\ast}(z)\in L^{\infty}_{M_2}$ where $\Phi_0$ is a constant matrix.
Using Theorem \ref{main},   we give a neccessary and sufficient condition of complex symmetric block Toeplitz operators with the special conjugation.

%3.4
\begin{thm}\label{inf-CSO-00} {\it 
 If  $T_{\varphi_j}  (j=1,2,3,4)$ are  complex symmetric with the conjugation  $C_{1,\lambda}$ where ${\varphi_1}={\varphi_2+\varphi_3+\varphi_4}$, then the following statements are equivalent:\\
{\em (i)} $T_{\Phi}$ is complex symmetric with  the conjugation ${\mathcal C}=
\frac{1}{\sqrt{2}}\left(
\begin{array}
[c]{cc}%
C_{1,\lambda} & C_{1,\lambda}\\
C_{1,\lambda} & -C_{1,\lambda}%
\end{array}
\right) $.\\
{\em (ii)} $\widehat{\Phi}(n)\lambda^n=\widehat{\Phi}(-n)$ for all $n\in{\Bbb Z}$ and $|\lambda|=1$.\\
{\em (iii)}  $\Phi(z)=\Phi_0+\sum_{n=1}^{\infty}\widehat{\Phi}(n)(z^n+\lambda^n\overline{z}^{n})$  for $|\lambda|=1$.\\
{\em (iv)} $\Phi(z)=\Phi_{+}(z)+\Phi_0+{\Phi_{+}(\lambda\overline{z})}$ for $\Phi_{+}\in zH^2_{{\Bbb C}^2}$ and $|\lambda|=1$. 
}
\end{thm}

{\it Proof.} By Theorem \ref{main} and Lemma \ref{lem_1}, $T_{\Phi}$ is complex symmetric with the conjugation ${\cal C}$ if and only if $T_{\varphi_j}(j=1,2,3,4)$ are complex symmetric with the conjugation $C_{1,\lambda}$.
 
 (i)$\Leftrightarrow$(ii):  
Since  $T_{\varphi_j}(j=1,2,3,4)$ are complex symmetric with  the conjugation $C_{1,\lambda}$ if and only if
 $\hat{\varphi_j}(-n)=\lambda\hat{\varphi_j}(n)$ for all $n\in{\Bbb Z}$ and for $j=1,2,3,4$  by  \cite{KL}, we conclude that
$$\widehat{\Phi}(n)\lambda^n=\widehat{\Phi}(-n)$$ for all $n\in{\Bbb Z}$  and $|\lambda|=1$.

 (i) $\Leftrightarrow$ (iii):
Let  $T_{\Phi}$ be complex symmetric with the conjugation ${\cal C}$. Then, from  assertion (ii),
\begin{eqnarray}\label{cso-cor}
\Phi(z)&=&\sum_{n=1}^{\infty}\widehat{\Phi}(n)z^n+\Phi_0+\sum_{n=1}^{\infty}\widehat{\Phi}(-n)\overline{z}^n\cr
&=&\sum_{n=1}^{\infty}\widehat{\Phi}(n)z^n+\Phi_0+\sum_{n=1}^{\infty}\widehat{\Phi}(n)\lambda^n\overline{z}^n\cr
&=&\Phi_0+\sum_{n=1}^{\infty}\widehat{\Phi}(n)(z^n+\lambda^n\overline{z}^n),
\end{eqnarray} and therefore we have the statement (iii).

Conversely, if $\Phi(z)=\Phi_0+\sum_{n=1}^{\infty}\widehat{\Phi}(n)(z^n+\lambda^n\overline{z}^n)$ with $|\lambda|=1$,
then (\ref{cso-cor}) and \cite{KL} implies that
\begin{eqnarray*}
\widehat{\Phi}(-n)&=&\begin{pmatrix}\langle \varphi_1, e_{-n}\rangle & \langle \varphi_2, e_{-n}\rangle\cr
\langle \varphi_3, e_{-n}\rangle & \langle \varphi_4, e_{-n}\rangle\end{pmatrix}\cr
&=&
\langle \sum_{n=1}^{\infty}\widehat{\Phi}(n)z^n+\Phi_0+\sum_{n=1}^{\infty}\widehat{\Phi}(n)\lambda^n\overline{z}^n, e_{-n}I_{M_2}\rangle\cr
&=&\lambda^n\widehat{\Phi}(n)
\end{eqnarray*} for all $n\in{\Bbb Z}$. By (i) $\Leftrightarrow$  (ii),  $T_{\Phi}$ is complex symmetric with  the conjugation ${\cal C}$.

(i) $\Leftrightarrow$ (iv): Since $T_{\Phi}$ is complex symmetric with the conjugation ${\cal C}$, it follows from (ii) that
 $$\Phi_{-}(z)=\sum_{n=1}^{\infty}\overline{\widehat{\Phi}(n)}\overline{\lambda}^n{z}^n=\overline{\sum_{n=1}^{\infty}\widehat{\Phi}(n)\lambda^n\overline{z}^n} ={\Phi_{+}^{\ast}(\lambda\overline{z})}$$   where $\ast$ denotes the complex conjugate.
Therefore, $\Phi(z)=\Phi_{+}(z)+\Phi_0+{\Phi_{+}(\lambda\overline{z})}$  with $|\lambda|=1$.

Conversely, assume that $\Phi(z)=\Phi_{+}(z)+\Phi_0+{\Phi_{+}(\lambda\overline{z})}$  with $|\lambda|=1$.
Then
\begin{eqnarray*}
\widehat{\Phi}(-n)&=& \langle [\Phi_{+}(z)+\Phi_0+{\Phi_{+}(\lambda\overline{z})}], e_{-n}I_{M_2} \rangle\cr
&=&\langle [\sum_{n=1}^{\infty}\widehat{\Phi}(n)z^n+\Phi_0+\sum_{n=1}^{\infty}\widehat{\Phi}(n)\lambda^n\overline{z}^n], e_{-n}I_{M_2}\rangle=\lambda\widehat{\Phi}(n)
\end{eqnarray*}  with $|\lambda|=1$.
  By (i) $\Leftrightarrow$  (ii), we know that $T_{\Phi}$ is complex symmetric with the conjugation ${\cal C}$. \pfend

\medskip

As an application of Theorem \ref{inf-CSO-00}, we recapture the following result.
%3.5
\begin{cor}(\cite{KL1})\label{G_1} {\it
 If  $T_{\varphi_j}  (j=1,2,3,4)$ are  complex symmetric with the conjugation  $C_{1,1}$ where ${\varphi_1}={\varphi_2+\varphi_3+\varphi_4}$, then the following statements are equivalent:\\
{\em (i)} $T_{\Phi}$ is complex symmetric with  the conjugation ${\mathcal C}$ {\em(}with $\lambda=1${\em )}.\\
{\em (ii)} $\widehat{\Phi}(n)=\widehat{\Phi}(-n)$ for all $n\in{\Bbb Z}$.\\
{\em (iii)}  $\Phi(z)=\Phi_0+\sum_{n=1}^{\infty}\widehat{\Phi}(n)(z^n+\overline{z}^{n})$.\\
{\em (iv)} $\Phi(z)=\Phi_{+}(z)+\Phi_0+{\Phi_{+}(\overline{z})}$ for $\Phi_{+}\in zH^2_{{\Bbb C}^2}$.
}
\end{cor}

\medskip

We say that a block Toeplitz operator $T_{\Phi}$ is {\it analytic} if $\Phi\in H^{\infty}_{M_2}$ and {\it coanlaytic}  if $\overline{\Phi}\in H^{\infty}_{M_2}$, respectively. If $M=\left(
\begin{array}
[c]{cc}%
T_1 & T_2\\
T_3 & T_4%
\end{array}
\right)$ where $T_j$ are bounded linear operators on a Hilbert space, then $det(M):=T_1T_4-T_2T_3$. 
%3.6
\begin{thm}\label{Toep-pro}
 {\it  For $\Phi\in L^{\infty}_{M_2},$ let $T_{\Phi}$ be a  complex symmetric operator on $H^{2}_{{\Bbb C}^2}$ and let $CK_{\mu}$ be inner and $det(CK_{\mu})$ be nonzero where $C$ is the conjugation on $H^{2}_{{\Bbb C}^2}$. 
If  $T_{\Phi}$ is analytic or coanalytic, then
$\Phi$ is  either identically zero on $\Bbb D$ or a nonzero constant function on ${\Bbb D}.$
}
\end{thm}

\smallskip

{\it Proof.} Suppose $T_\Phi$ is analytic.
For $\mu \in{\Bbb D}$, let $K_\mu= k_\mu I_{M_2}$, where $k_\mu= \frac{1}{1-\overline \mu z}$ is the reproducing kernel of $H^2$ at $\mu \in \mathbb D$. Observe for any $f\in H^{\infty}$ and $g \in H^2$, 
\begin{equation}
\langle g, T_{\overline f}k_\mu \rangle= \langle T_f g, k_\mu \rangle =\langle fg , k_\mu \rangle = f(\mu)g(\mu)= \langle g, \overline{f(\mu)}k_\mu \rangle,
\end{equation}
which implies that $T_{\overline f}k_\mu=\overline{f(\mu)}k_\mu$.
Now it is easy to verify $T_{\Phi}^{\ast}K_{\mu }= \overline{\Phi(\mu)} K_\mu$.
 Observe
\begin{eqnarray}\label{thm-E1}
CT_{\Phi}^{\ast}K_{\mu }- T_{\Phi}{C}K_{\mu}
& = & {C}\overline{\Phi(\mu)}K_{\mu}-T_{\Phi}{ C}K_{\mu }\cr
&=&{\Phi(\mu)}{C}K_{\mu}-P({\Phi}{C}K_{\mu})=[{\Phi(\mu)}-{\Phi}]{ C}K_{\mu}=0.
\end{eqnarray}
Therefore, we find that ${\Phi(\mu)}-{\Phi}$ is a zero (matrix) function, that is, $\Phi$ actually is a constant matrix. 

On the other hand,  if $T_{\Phi}$ is coanalytic, then  $\overline{\Phi}\in H^{\infty}_{M_2}$.
Hence $\Phi$ must be a constant function or $\Phi=0$ by a similar argument.
 \pfend

\medskip

We observe that if $T_{\Phi}$ is complex symmetric where $\Phi\in H^{\infty}_{M_2}$, then  it must be normal from \cite[Lemma 3.1]{WG}.
We say that  $\Phi$ is a nonconstant inner function on ${\Bbb D}$  which means that $\Phi^{\ast}\Phi=I$ and the entries of $\Phi$ are nonconstant functions.
%3.7
\begin{cor}{\it Let   $CK_{\mu}$ be inner and let $det(CK_{\mu})$ be nonzero where $C$ is the conjugation on $H^{2}_{{\Bbb C}^2}$ and $K_{\mu}=k_{\mu}I_{M_2}$. 
 If $\Phi$ is a nonconstant inner function on ${\Bbb D},$  then
$T_{\Phi}$ is not a complex symmetric operator.}
\end{cor}

\smallskip

{\it Proof.} The proof immediately follows from Theorem \ref{Toep-pro}.
\pfend

\medskip 
Next, we consider the relations among of the Fourier coefficients of $\varphi_j (j=1,2,3,4)$ for complex symmetric block Toeplitz operators $T_{\Phi}$. 
%3.8
\begin{prop}\label{cor-general}{\it
 If  $T_{\Phi}$ is a complex symmetric operator with the conjugation
 ${\mathcal C}=\frac{1}{\sqrt{2}}\left(
\begin{array}
[c]{cc}%
C_{1,\lambda} & C_{1,\lambda}\\
C_{1,\lambda} & -C_{1,\lambda}%
\end{array}
\right) $, then the following identities hold.
\begin{eqnarray*}
(a)&&\sum_{i=0}^{k-1}[(\hat{\varphi}_1+\hat{\varphi}_3)(k-i)-\overline{\lambda}^{(k+i)}({\hat{\varphi}_1}+{\hat{\varphi}_3})(-(k-i))]a_i\cr
&=&\sum_{n=1}^{\infty}[\overline{\lambda}^{(n+2k)}({\hat{\varphi}_1}+{\hat{\varphi}_3})(n)-(\hat{\varphi}_1+\hat{\varphi}_3)(-n)]a_{n+k},
\end{eqnarray*}
 \begin{eqnarray*}
(b)&&\sum_{i=0}^{k-1}[(\hat{\varphi}_2-\hat{\varphi}_4)(k-i)-\overline{\lambda}^{(k+i)}({\hat{\varphi}_2}-{\hat{\varphi}_4})(-(k-i))]b_i\cr
&=&\sum_{n=1}^{\infty}[\overline{\lambda}^{(n+2k)}({\hat{\varphi}_2}-{\hat{\varphi}_4})(n)-(\hat{\varphi}_2-\hat{\varphi}_4)(-n)]b_{n+k},
\end{eqnarray*}
 \begin{eqnarray*}
(c)&&\sum_{i=0}^{k-1}[(\hat{\varphi}_2+\hat{\varphi}_4)(k-i)-\overline{\lambda}^{(k+i)}({\hat{\varphi}_1}-{\hat{\varphi}_3})(-(k-i))]b_i\cr
&=&\sum_{n=1}^{\infty}[\overline{\lambda}^{(n+2k)}({\hat{\varphi}_1}-{\hat{\varphi}_3})(n)-(\hat{\varphi}_2+\hat{\varphi}_4)(-n)]b_{n+k},
\end{eqnarray*}
 \begin{eqnarray*}
(d)&&\sum_{i=0}^{k-1}[(\hat{\varphi}_1-\hat{\varphi}_3)(k-i)-\overline{\lambda}^{(k+i)}({\hat{\varphi}_2}+{\hat{\varphi}_4})(-(k-i))]a_i\cr
&=&\sum_{n=1}^{\infty}[\overline{\lambda}^{(n+2k)}({\hat{\varphi}_2}+{\hat{\varphi}_4})(n)-(\hat{\varphi}_1-\hat{\varphi}_3)(-n)]a_{n+k}
\end{eqnarray*}  for  all $k$
 where $|\lambda|=1.$}
\end{prop}

{\it Proof.}  Suppose that $C$ is the conjugation on $H^2$.
 Let $h(z)=\sum_{k=1}^{\infty}\alpha_kz^k$ and $g(z)=\sum_{k=1}^{\infty}\beta_kz^k,$ where
$$\alpha_k=\sum_{n=1}^{\infty}\widetilde{a_{n+k}}\overline{\hat{\varphi}(n)}\ \mbox{and}\ \beta_k=\sum_{i=0}^{k-1}\overline{\hat{\varphi}(-(k-i))}\widetilde{a_i}.$$
Then $Ch(z)=\sum_{k=1}^{\infty}\widetilde{\alpha_k}z^k$ and $Cg(z)=\sum_{k=1}^{\infty}\widetilde{\beta_k}z^k$.
If $C=C_{\mu,\lambda}$, then $\widetilde{a_j}=\mu\lambda^j \overline{a_j}$ with $|\mu|=|\lambda|=1$ and hence for any $\varphi$, we have
$${\alpha_k}=\sum_{n=1}^{\infty}\mu{\lambda}^{n+k}\overline{a_{n+k}}\overline{\hat{\varphi}(n)}\  \mbox{and}\
{\beta_k}=\sum_{i=0}^{k-1}\mu{\lambda}^i\overline{\hat{\varphi}(-(k-i))}\overline{a_i}.$$
Then
\begin{eqnarray}\label{F_1}
 \widetilde{\alpha_k}&=&\langle (C_{\mu,\lambda}\sum_{k=1}^{\infty}[\sum_{n=1}^{\infty}\mu{\lambda}^{n+k}\overline{a_{n+k}}\overline{\hat{\varphi}(n)}]z^k), z^k \rangle\cr
&=&\langle C_{\mu,\lambda}z^k, \sum_{k=1}^{\infty}[\sum_{n=1}^{\infty}\mu{\lambda}^{n+k}\overline{a_{n+k}}\overline{\hat{\varphi}(n)}]z^k \rangle\cr
&=&\langle \mu \overline{\lambda}^{k}z^k, \sum_{k=1}^{\infty}[\sum_{n=1}^{\infty}\mu{\lambda}^{n+k}\overline{a_{n+k}}\overline{\hat{\varphi}(n)}]z^k \rangle
=\sum_{n=1}^{\infty}\overline{\lambda}^{(n+2k)}{a_{n+k}}{\hat{\varphi}(n)}
\end{eqnarray} and, similarly,
\begin{equation}\label{F_2}
\widetilde{\beta_k}=\sum_{i=0}^{k-1}\overline{\lambda}^{(k+i)}{\hat{\varphi}(-(k-i))}{a_i}.
\end{equation}
 Now, by  using (\ref{F_1}) and (\ref{F_2}),  we obtain that
\begin{eqnarray*}
&&\sum_{i=0}^{k-1}[(\hat{\varphi}_1+\hat{\varphi}_3)(k-i)-\overline{\lambda}^{(k+i)}({\hat{\varphi}_1}+{\hat{\varphi}_3})(-(k-i))]a_i\cr
&=&\sum_{n=1}^{\infty}[\overline{\lambda}^{(n+2k)}({\hat{\varphi}_1}+{\hat{\varphi}_3})(n)-(\hat{\varphi}_1+\hat{\varphi}_3)(-n)]a_{n+k},
\end{eqnarray*}
 where $|\lambda|=1.$
 Similarly, we get that
 \begin{eqnarray*}
&&\sum_{i=0}^{k-1}[(\hat{\varphi}_2-\hat{\varphi}_4)(k-i)-\overline{\lambda}^{(k+i)}({\hat{\varphi}_2}-{\hat{\varphi}_4})(-(k-i))]b_i\cr
&=&\sum_{n=1}^{\infty}[\overline{\lambda}^{(n+2k)}({\hat{\varphi}_2}-{\hat{\varphi}_4})(n)-(\hat{\varphi}_2-\hat{\varphi}_4)(-n)]b_{n+k},
\end{eqnarray*}
 \begin{eqnarray*}
&&\sum_{i=0}^{k-1}[(\hat{\varphi}_2+\hat{\varphi}_4)(k-i)-\overline{\lambda}^{(k+i)}({\hat{\varphi}_1}-{\hat{\varphi}_3})(-(k-i))]b_i\cr
&=&\sum_{n=1}^{\infty}[\overline{\lambda}^{(n+2k)}({\hat{\varphi}_1}-{\hat{\varphi}_3})(n)-(\hat{\varphi}_2+\hat{\varphi}_4)(-n)]b_{n+k},
\end{eqnarray*}
 \begin{eqnarray*}
&&\sum_{i=0}^{k-1}[(\hat{\varphi}_1-\hat{\varphi}_3)(k-i)-\overline{\lambda}^{(k+i)}({\hat{\varphi}_2}+{\hat{\varphi}_4})(-(k-i))]a_i\cr
&=&\sum_{n=1}^{\infty}[\overline{\lambda}^{(n+2k)}({\hat{\varphi}_2}+{\hat{\varphi}_4})(n)-(\hat{\varphi}_1-\hat{\varphi}_3)(-n)]a_{n+k},
\end{eqnarray*}  for  all $k$ where $|\lambda|=1.$
Since $T_{\Phi}$ is complex symmetric with the conjugation
 ${\mathcal C}$, we know that  (a),(b),(c), and (d) always hold.
\pfend

%3.9
\begin{cor}\label{cor-g}
 Let ${\Phi}=\begin{pmatrix}{\varphi} & {\psi}\cr
-{\varphi} & {\psi}\end{pmatrix}$.
 If  $T_{\Phi}$ is a complex symmetric operator with the conjugation
 ${\mathcal C}=\frac{1}{\sqrt{2}}\left(
\begin{array}
[c]{cc}%
C_{1,\lambda} & C_{1,\lambda}\\
C_{1,\lambda} & -C_{1,\lambda}%
\end{array}
\right) $, then
$$
(1) \sum_{i=0}^{k-1}2[\hat{\psi}(k-i)-\overline{\lambda}^{(k+i)}{\hat{\varphi}(-(k-i))}]a_i
=\sum_{n=1}^{\infty}2[\overline{\lambda}^{(n+2k)}{\hat{\varphi}(n)}-\hat{\psi}(-n)]a_{n+k},
$$
$$
(2) \sum_{i=0}^{k-1}2[\hat{\varphi}(k-i)-\overline{\lambda}^{(k+i)}{\hat{\psi}(-(k-i))}]b_i
=\sum_{n=1}^{\infty}2[\overline{\lambda}^{(n+2k)}{\hat{\psi}(n)}-\hat{\varphi}(-n)]b_{n+k}
.$$
\end{cor}

{\it Proof.} The proof follows from Proposition \ref{cor-general}. \pfend
\medskip

%3.10
\begin{cor}
 Let ${\Phi}=\begin{pmatrix}{\varphi} & {\varphi}\cr
-{\varphi} & {\varphi}\end{pmatrix}$.
 If  $T_{\Phi}$ is a complex symmetric operator with the conjugation
 ${\mathcal C}=\frac{1}{\sqrt{2}}\left(
\begin{array}
[c]{cc}%
C_{1,\lambda} & C_{1,\lambda}\\
C_{1,\lambda} & -C_{1,\lambda}%
\end{array}
\right) $, then
$$
\sum_{i=0}^{k-1}2[\hat{\varphi}(k-i)-\overline{\lambda}^{(k+i)}{\hat{\varphi}(-(k-i))}]a_i
=\sum_{n=1}^{\infty}2[\overline{\lambda}^{(n+2k)}{\hat{\varphi}(n)}-\hat{\varphi}(-n)]a_{n+k},
$$
\end{cor}

{\it Proof.} The proof follows from Corollary \ref{cor-g}. \pfend

\medskip

\section{ Complex symmetry and normality of a block Toeplitz operator}

Let ${\cal{L}}({\cal{H}})$ denote the algebra of all bounded linear operators on a separable complex Hilbert space ${\cal{H}}$.
In this section, we study normal  block Toeplitz operators as some examples of complex symmetric block Toeplit operators with some conjugation ${\mathcal C
}$.
For this, we need the following lemma.
%4.1
\begin{lem}
(\cite[Theorem 2.1]{KKL}) \label{normal}%2.1
{\it For $T=\begin{pmatrix} A & B\cr
 C  & D \end{pmatrix}\in{\cal L}({\cal H}\oplus {\cal H})$, let
 $$
\begin{cases}
 t_1=A^*A+C^*C, t_2=A^*B+C^*D, t_3=B^*B+D^*D,\\
s_1=AA^*+BB^*, s_2=AC^*+BD^*, s_3=CC^*+DD^*.
\end{cases}
$$
 Then $T$ is normal if and only if $t_{j}=s_{j}$ for all $1 \leq j \leq 3$.
}
\end{lem}
\medskip

In general, even if $A,B,C,D$ are normal, then $T$ may not be normal. So we investigate the necessary and sufficient conditions so that the $2 \times 2$ operator matrix $T$ to be normal.
%4.2
\begin{lem}
\label{normal_1}{\it
For $R,S\in{\mathcal L}(\mathcal H),$ put $[R,S]=RS-SR.$\\
{\em (i)} For $T=\begin{pmatrix} A & B\cr
 B  & D \end{pmatrix}\in{\cal L}({\cal H}\oplus {\cal H})$, if $A$, $B$, and  $D$ are normal, then  $T$  is normal if and only if $A^*B+B^*D=AB^*+BD^*$ holds.\\
{\em (ii)}  For $T=\begin{pmatrix} A & B\cr
 C  & A \end{pmatrix}\in{\cal L}({\cal H}\oplus {\cal H})$, if $A$ is normal and $[A,B]=[A,C]=0$, then  $T$  is normal if and only if $B^*B=CC^{\ast}$ and $BB^*=C^{\ast}C$ holds.\\
{\em (iii)} For $T=\begin{pmatrix} A & B\cr
 B^{\ast}  & A \end{pmatrix}\in{\cal L}({\cal H}\oplus {\cal H})$,  $T$  is normal if and only if $A$ and $D$ are normal and $AB=BD$ holds.\\
{\em (iv)} For $T=\begin{pmatrix} A & B\cr
 B  & A \end{pmatrix}\in{\cal L}({\cal H}\oplus {\cal H})$,  $T$  is normal if and only if $A$ and $B$ are normal and $[A,B]=0$ holds.
}
\end{lem}

{\it Proof.}
(ii) Assume that $A$ is normal, $AB=BA$, and $AC=CA$. By the Fuglede-Putnam theorem, $A^{\ast}B=BA^{\ast}$ and  $A^*C=CA^{\ast}$ (see \cite{Fug}).
Since $$t_{1}=A^{\ast}A+C^{\ast}C, t_{2}=A^{\ast}B+C^{\ast}A, t_{3}= B^{\ast}B+A^{\ast}A,$$ and  $$s_{1}=AA^{\ast}+BB^{\ast}, s_{2}=AC^{\ast}+BA^{\ast}, s_{3}= CC^{\ast}+AA^{\ast},$$
it follows that $t_{2}={s}_{2}$. Hence  $T$  is normal if and only if $B^*B=CC^{\ast}$ and $BB^*=C^{\ast}C$ holds.

The statements (i), (iii), and (iv) hold by a similar way.
\pfend 

\medskip

%4.3
\begin{prop}\label{nor-1} {\it  
 Let  $\varphi_3=\varphi_2$ and $\varphi_j=\alpha_j+\beta_j\rho_j$ where  $\rho_j$ is a real-valued function in $L^{\infty}$ and $\alpha_j, \beta_j\in{\Bbb C}$ for $j=1,2,4$.
Assume that  $\widehat{\Phi}(n)=\widehat{\Phi}(-n)$ for all $n\in{\Bbb Z}$ and ${\varphi_1}={2\varphi_2+\varphi_4}$.
Then the following properties hold.\\
{\em (i)} If   $T_{\varphi_1}$ and $T_{\varphi_2}$ commute,  then $T_{\Phi}$ is normal  complex symmetric with the conjugation ${\mathcal C}${\em (}with $\lambda=1${\em)}.  \\
{\em (ii)}  If $D^{\ast}\not=-D$ where
$D:=T_{\varphi_1}T_{\varphi_2}^{\ast}-T_{\varphi_2}^{\ast}T_{\varphi_1}$,  then $T_{\Phi}$ is nonnormal  complex symmetric with the conjugation  ${\mathcal C}${\em(}with $\lambda=1${\em)}.}
\end{prop}

{\it Proof.}  By Corollary \ref{G_1}, we know that $T_{\Phi}$ is  complex symmetric with the conjugation  ${\cal C}$ (with $\lambda=1$). By hypothesis, $T_{\varphi_j}$ are normal for $j=1,2,4$ from \cite{BH}. 

(i)
Since $T_{\varphi_1}T_{\varphi_2}=T_{\varphi_2}T_{\varphi_1}$ and ${\varphi_4}=\varphi_1-{2\varphi_2}$, it follows from  Fuglede-Putnam theorem that
$$ T_{\varphi_1}^{\ast} T_{\varphi_2}+ T_{\varphi_2}^{\ast}T_{\varphi_4}= T_{\varphi_1} T_{\varphi_2}^{\ast}+ T_{\varphi_2} T_{\varphi_4}^{\ast}.$$
Hence $T_{\Phi}$ is normal from Lemma \ref{normal_1}(i).

(ii) Let $D^{\ast}\not=-D$ where
$D=T_{\varphi_1}T_{\varphi_2}^{\ast}-T_{\varphi_2}^{\ast}T_{\varphi_1}$.
Since ${\varphi_4}=\varphi_1-{2\varphi_2}$ and  $T_{\varphi_2}$ is normal,  it follows that 
\begin{eqnarray*}
&&T_{\varphi_1}^{\ast} T_{\varphi_2}+ T_{\varphi_2}^{\ast}T_{\varphi_4}-T_{\varphi_1} T_{\varphi_2}^{\ast}-T_{\varphi_2} T_{\varphi_4}^{\ast}\cr
&=&T_{\varphi_1}^{\ast} T_{\varphi_2}+ T_{\varphi_2}^{\ast}T_{(\varphi_1-2\varphi_2)}-T_{\varphi_1} T_{\varphi_2}^{\ast}-T_{\varphi_2} T_{(\varphi_1-2\varphi_2)}^{\ast}\cr
&=&T_{\varphi_1}^{\ast} T_{\varphi_2}+ T_{\varphi_2}^{\ast}T_{\varphi_1}-T_{\varphi_1} T_{\varphi_2}^{\ast}-T_{\varphi_2} T_{\varphi_1}^{\ast}\not=0.
\end{eqnarray*}
Hence  $T_{\Phi}$ is not normal  from Lemma \ref{normal_1}(i).
\pfend

\medskip

%4.4
\begin{prop}\label{nor-2} 
{\em Let  ${\varphi_{3}}=\overline{\varphi_{2}}$ and  $\varphi_4=\varphi_1$ and $\varphi_j=\alpha_j+\beta_j\rho_j$ where  $\rho_j$ is a real-valued function in $L^{\infty}$ and $\alpha_j, \beta_j\in{\Bbb C}$ for $j=1,2$.
Assume that  $\widehat{\Phi}(n)=\widehat{\Phi}(-n)$ for all $n\in{\Bbb Z}$ and  $Re (\varphi_2)=0$. Then the following properties hold.\\
{\em (i)} If $T_{\varphi_1}$ and $T_{\varphi_2}$   commute,  then $T_{\Phi}$ is normal  complex symmetric with the conjugation   ${\mathcal C}$ {\em(}with $\lambda=1${\em)}.  \\
{\em (ii) }   If $D^{\ast}\not=-D$ where
$D:=T_{\varphi_1}T_{\varphi_2}^{\ast}-T_{\varphi_2}^{\ast}T_{\varphi_1}$,  then $T_{\Phi}$ is nonnormal  complex symmetric with the conjugation  ${\mathcal C}$ {\em(}with $\lambda=1${\em)}.}
\end{prop}

{\it Proof.} By Corollary \ref{G_1}, we know that $T_{\Phi}$ is  complex symmetric with the conjugation  ${\cal C}$ (with $\lambda=1$). By hypothesis, $T_{\varphi_j}$ are normal for $j=1,2$ from \cite{BH}.

 (i) 
Since $T_{\varphi_1}T_{\varphi_2}=T_{\varphi_2}T_{\varphi_1}$, it follows from  Fuglede-Putnam theorem that
$$ T_{\varphi_1}^{\ast} T_{\varphi_2}+ T_{\varphi_2}^{\ast}T_{\varphi_1}= T_{\varphi_1} T_{\varphi_2}^{\ast}+ T_{\varphi_2} T_{\varphi_1}^{\ast}.$$
Therefore $T_{\Phi}$ is  normal from Lemma \ref{normal_1}(i). 

(ii)    If $D^{\ast}\not=-D$ where
$D=T_{\varphi_1}T_{\varphi_2}^{\ast}-T_{\varphi_2}^{\ast}T_{\varphi_1}$,
then $$T_{\varphi_1}^{\ast} T_{\varphi_2}+ T_{\varphi_2}^{\ast}T_{\varphi_1}-T_{\varphi_1} T_{\varphi_2}^{\ast}-T_{\varphi_2} T_{\varphi_1}^{\ast}\not=0.$$
Hence $T_{\Phi}$ is not normal  from Lemma \ref{normal_1}(i).
%The proof  follows from (i) and Lemma \ref{normal_1}.
\pfend

\medskip
%4.5
\begin{lem}\label{GHR}(\cite{GHR}, 2011, Gu-Hendricks-Rutherford)
{\it Let $\Phi=\Phi_{+}+\Phi_0+\Phi_{-}^{\ast}\in L^{\infty}_{M_2}$.
Assume that $det(\Phi_{+})$ is not identically zero.
Then $T_{\Phi}$ is normal if and only if
 $\Phi^{\ast}\Phi=\Phi\Phi^{\ast}$ and
 $\Phi_{+}=\Phi_{-}U$ for some constant unitary matrix $U$.}
\end{lem}
%\medskip
%4.6
\begin{prop}
 {\it
 Assume that $T_{\varphi_j}  (j=1,2,3,4)$ are  complex symmetric with the conjugation  $C_{1,1}$ where ${\varphi_1}={\varphi_2+\varphi_3+\varphi_4}$. 
If  $$|\varphi_2|=|\varphi_3|\ \mbox{and}\
(\varphi_1-\varphi_4)\overline{\varphi_3}=(\overline{\varphi_1}-\overline{\varphi_4})\varphi_2$$  and ${\Phi_{-}}^{\ast}(z)=\Phi_{+}(\overline{z})=U\Phi_{+}^{\ast}(z)$ for some constant unitary matrix $U$,
then $T_{\Phi}$ is normal complex symmetric with  the conjugation ${\mathcal C}$ {\em (}with $\lambda=1${\em)}.
}
\end{prop}

{\it Proof.} Since note that   $\Phi^{\ast}\Phi=\Phi\Phi^{\ast}$ if  and only if $$|\varphi_2|=|\varphi_3|\ \mbox{and}\
(\varphi_1-\varphi_4)\overline{\varphi_3}=(\overline{\varphi_1}-\overline{\varphi_4})\varphi_2,$$ 
the proof follows from Corollary \ref{G_1} and Lemma \ref{GHR}.\pfend

\bigskip

%\noindent Caixing Gu

%\noindent Department of Mathematics, California Polytechnic State University, San Luis
%Obispo, CA 93407

%\noindent e-mail: cgu@calpoly.edu

%\bigskip

\noindent Dong-O Kang

\noindent Department of Mathematics, Chungnam National University, Daejeon, 305-764, Korea

\noindent E-mail: dokang@cnu.ac.kr

\vspace{5mm}
\noindent Eungil Ko

\noindent Department of Mathematics, Ewha Womans University, Seoul 120-750, Korea

\noindent e-mail: eiko@ewha.ac.kr

\vspace{5mm}

\noindent Ji Eun Lee

\noindent Department of Mathematics and Statistics, Sejong University, Seoul 143-747, Korea

\noindent e-mail: jieunlee7@sejong.ac.kr; jieun7@ewhain.net
\end{document}